\documentclass[12pt]{amsart}

\usepackage{amssymb,amsmath,latexsym,amsthm}
\usepackage[english]{babel}
\usepackage[all]{xy}

\newtheorem{theorem}{Theorem}[section]
\newtheorem{lemma}{Lemma}[section]
\newtheorem{proposition}{Proposition}[section]
\newtheorem{cor}{Corollary}[section]

\theoremstyle{definition}

\newtheorem*{remark}{Remark}
\newtheorem{example}{Example}

\begin{document}

\title{On optimal embeddings and trees}


\author{\sc Manuel Arenas\\
\sc Luis Arenas-Carmona\\
\sc Jaime Contreras\\}


\newcommand\Q{\mathbb Q}
\newcommand{\h}{\mathfrak{H}}
\newcommand\alge{\mathfrak{A}}
\newcommand\Da{\mathfrak{D}}
\newcommand\Ea{\mathfrak{E}}
\newcommand\Ha{\mathfrak{H}}
\newcommand\oink{\mathcal O}
\newcommand\oinki{\mathcal O_k}
\newcommand\matrici{\mathbb{M}}
\newcommand\Txi{\lceil}
\newcommand\ad{\mathbb{A}}
\newcommand\enteri{\mathbb Z}
\newcommand\finitum{\mathbb{F}}
\newcommand\bbmatrix[4]{\left(\begin{array}{cc}#1&#2\\#3&#4\end{array}\right)}
\newcommand\lbmatrix[4]{\textnormal{\scriptsize{$\left(\begin{array}{cc}#1&#2\\#3&#4
\end{array}\right)$}}\normalsize}
\newcommand\lcvector[2]{\textnormal{\scriptsize{$\left(\begin{array}{c}#1\\#2\end{array}\right)$}}\normalsize}

\maketitle

\begin{abstract}
We apply the theory of Bruhat-Tits trees to the study of optimal embeddings of two and three dimensional commutative orders into quaternion algebras. Specifically,  we determine how many conjugacy classes of global Eichler orders in a quaternion algebra yield optimal representations of such orders. This completes the previous work by  C. Maclachlan, who considered only  Eichler orders of square free level and integral domains as sub-orders. The same technique is used in the second part of this work to compute local embedding numbers, extending previous results by J. Brzezinski.
\end{abstract}

\bigskip
\section{Introduction}

Let $K$ be a number  field, let $\oink$ be the ring of integers of $K$, and let $\alge$ be an indefinite
 quaternion $K$-algebra. Let
$L\subseteq\alge$ be a subalgebra, and let $\Ha$ be an order of full rank in $L$. The question on whether every maximal order $\Da\subseteq\alge$ contains a conjugate or, as we say in what follows, represents $\Ha$, is known as the selectivity
problem. When $L$ is a maximal commutative sub-algebra, the conditions on $\alge$ and $\Ha$ for which selectivity
can occur were described completely by T. Chinburg and E. Friedman in \cite{FriedmannQ}.
 These results where extended to Eichler orders $\Da$ in \cite{Guo} and \cite{Chan}.
  The second author of the present work gave a generalization to representations of 
an arbitrary suborder $\Ha$ into into finite intersections of maximal orders \cite{eichler2}.
 A different generalization of the result in  \cite{FriedmannQ} to a large family of orders 
of maximal rank was given by Linowitz \cite{lino1}. These results are partly motivated by the rol played
by quaternion orders in some early constructions of isospectral, but non-isometric hyperbolic  varieties.
This is possible since, in the absence of selectivity, the number of essentially different ways to embed
$\Ha$ into $\Da$, the global embedding number, remains constant as $\Da$ runs over a genus of orders,
see \cite{Vigneras2} for details. 

 By definition, an embedding $\phi:L\rightarrow\alge$ induces an optimal embeding $\tilde{\phi}:\Ha\rightarrow\Da$ if $\phi^{-1}(\Da)=\Ha$. If $\Ha$ is a suborder of $\Da$, we say that $\Ha$ is optimal in $\Da$ or that $\Da$ optimally contains
$\Ha$ whenever the inclusion $i:\Ha\rightarrow \Da$ is optimal. 
 C. Maclachlan \cite{Macla} proved a result analogous to those in \cite{Guo} and \cite{Chan}, for optimal embeddings, provided that the level of the Eichler order $\Da$ is square free, a condition that is removed in the present work. Furthermore, Theorem 1.1 bellow applies to all non-trivial orders of non-full rank in a matrix algebra.  A crucial tool in our proof is the local interpretation of optimality in terms of branches (\S3). The same 
tool is used in subsequent sections to the study of local embeding numbers (Theorems 1.2-1.4), which generalize previous computations by J. Brzezinski \cite{enumbers}.

For any field $K$ whose characteristic is not $2$, the Hilbert symbol $\left(\frac{a,b}K\right)$ denotes 
the quaternion algebra $K[i,j]$ defined by the relations $i^2=a$, $j^2=b$, and $ij=-ji$.
The level of an Eichler order $\Ea=\Da_1\cap\Da_2$, where $\Da_1$ and $\Da_2$ are maximal, is an ideal
$I=\prod_\wp\wp^{\alpha_\wp}$ satisfying $\oink/I\cong\Da_1/\Ea\cong\Da_2/\Ea$ as $\oink$-modules. 
The non-negative integer $\alpha_\wp$, where $\wp$ is a maximal order of $\oink$, is called the local level of $\Ea$
at $\wp$. The level $I$ is a complete invariant for the genus $\mathrm{gen}(\Ea)$. The latter is denoted
 $\mathbb{O}_I$ in all that follows.

\begin{theorem}\label{t1}
Let $\alge$ be an indefinite quaternion algebra over a number field $K$, let $\Ha\subseteq\alge$ be an order of rank 2, and let $L=K\Ha$.
Let $\mathbb{O}_I$ be a genus of Eichler orders in $\alge$ representing $\Ha$, and set 
$I=\prod_\wp\wp^{\alpha_\wp}$ as above. Then:
\begin{enumerate}
\item If $L\cong K\times K$,then $\Ha$ embeds optimally into every order in $\mathbb{O}_I$.
\item If $L\cong K[x]/(x^2)$,  then $\Ha$ embeds optimally into the orders in exactly one 
conjugacy class in $\mathbb{O}_I$. 
\item If $L\cong K[\sqrt d]$ is a field, then $\Ha$ embeds optimally into every order in $\mathbb{O}_I$
 unless the following conditions are satisfied:
\begin{enumerate}
\item $\alge\cong\left(\frac{-1,d}K\right)$,
\item $L/K$ is unramified at all finite places, and
\item $L/K$ splits at $\wp$ whenever $\alpha_\wp$ is odd.
\end{enumerate}
Furthermore, if these conditions are satisfied then $\Ha$ embeds optimally into the orders in exactly one half of all conjugacy classes in $\mathbb{O}_I$. 
\end{enumerate}
On the other hand, if $L$ has rank 3, and if there is an order of $\mathbb{O}_I$ representing $\Ha$ optimally,
 then $\Ha$ embeds optimally into the orders in exactly one  conjugacy class in that genus. 
\end{theorem}

As in  \cite{Macla}, when optimal selectivity does occur, we describe the set of orders optimally representing $\Ha$,
 in terms of a Galois group valued distance between two isomorphism classes of orders (\S2). 
Note that condition (3b) implies that
 the quaternion algebra  $\left(\frac{-1,d}{K_\wp}\right)$ splits at every finite place $\wp$ \cite{ohm}. We conclude that,
 when $I$ is square free, the conditions in (3) reduce to the selectivity conditions
described in \cite{Macla}. When $L$ is  three-dimensional, the condition for existence of one optimal embedding
is given in local terms in Theorem 1.3 below. 

To fix ideas, let $k$ be a local field with ring of integers $\oinki$, let $\Ea\subseteq\matrici_2(k)$ be an
 Eichler order,  and let $\Ha$ be a suborder of $\Ea$. Let $X$ be the set of optimal embeddings 
$\phi:\Ha\rightarrow\Ea$ and let 
$Y$ be the set of optimal suborders of $\Ea$ that are isomorphic to $\Ha$. 
Let $\Gamma_1=k^*\Ea^*$, where $A^*$ denotes the group of units of a ring $A$,
 and let $\Gamma_2$ be the normalizer\footnote{In this work, normalizer is used as a synonym of stabilizer by 
conjugation, since $\Ea$ is not a subgroup of $\mathrm{GL}_2(k)$.} of $\Ea$ in $\mathrm{GL}_2(k)$.
 It is well known that the quotient $\Gamma_2/\Gamma_1$ has one element if
$\Ea$ is maximal and two elements otherwise. 
By the embedding number of $\Ha$ into $\Ea$ we mean any of the following quantities:
$$e_1=|X/\Gamma_1|,\quad e_2=|X/\Gamma_2|,\quad e_3=|Y/\Gamma_1|,\quad e_4=|Y/\Gamma_2|.$$
We use the vector $\stackrel{\rightarrow}e=(e_1,e_2,e_3,e_4)$ to simplify the statements below.

For any order $\Ha$, and any possitive integer $t$, we define 
$\Ha^{[t]}=\oinki\mathfrak{1}+\pi^t\Ha$, where $\oinki\mathfrak{1}\cong \oinki$
is the ring of integral scalar matrices and $\pi$ is a uniformizing parameter.
The complete list of orders in $\alge$ that are intersections of maximal orders is as follows 
\cite[Thm. 1.4]{eichler2}:
\begin{enumerate}
\item The ring $\oinki\mathfrak{1}$ itself.
\item The order generated by a nilpotent element.
\item An order spanning an algebra isomorphic to $k\times k$. These orders have the form $\Ha\cong(\oinki\times\oinki)^{[t]}$.
\item An order of the form $\Ha\cong\Ha_0^{[t]}$, where $t\geq0$ and
$\Ha_0=\lbmatrix {\oinki}{\oinki}0{\oinki}$ is the ring of integral upper triangular matrices.
\item An order of the form $\Ha\cong\Ea^{[t]}$, where $\Ea$ is an eichler order.
\end{enumerate}
As there is no non-trivial optimal embedding of an order of rank 4 into another, 
we do not consider the last case in what follows.
Throughout, we let $p$ be the cardinality of the residue field $\mathbb{K}=\oinki/\pi\oinki$, while 
$[t]$ is the largest integer not exceeding $t$.

\begin{theorem}\label{tb2}
Let $\Ea$ be an Eichler order of level $r$ and let $\Ha$ be the rank-2 order spanned by a nilpotent element. Then
$e_4=\left[\frac{r+2}2\right]$, $e_3=r+1$, and $e_1=p^{[r/2]}+p^{[(r-1)/2]}$, unless $r=0$ where $e_1=1$. Furthermore, $e_2=p^{[r/2]}$
if $r$ is even and $2\equiv 0\ (\textnormal{mod }\pi^{r/2})$, while $e_2=e_1/2$ otherwise.
\end{theorem}

\begin{theorem}\label{tb3}
If $\Ea$ is an Eichler order of level $r$, the order $\Ha=\lbmatrix \oinki\oinki0\oinki^{[t]}$ embeds optimally into $\Ea$ if and only if $r\geq 2t$ and in this case
$$\stackrel{\rightarrow}e=
\left\{\begin{array}{rcl}\Big((p-1)p^{2t-1},\frac12(p-1)p^{2t-1},1,1\Big)&\textnormal{ if }&r=2t\\
\Big(2(p-1)p^{2t-1},(p-1)p^{2t-1},2,1\Big)&\textnormal{ if }&r>2t\end{array}\right.,$$
unless $t=0$, where $\stackrel{\rightarrow}e=(1,1,1,1)$ if $r=0$ and $\stackrel{\rightarrow}e=(2,1,2,1)$ otherwise.
\end{theorem}

\begin{theorem}\label{tb1}
Let $\Ea$ be an Eichler order of level $r>0$ and let $\Ha\subseteq\Ea$ be an order isomorphic to $(\oinki\times\oinki)^{[t]}$.  For any triple $(r,u,t)\in(\mathbb{Z}_{\geq0})^3$ 
satisfying $v\leq u\leq [r/2]$ for $v=\max\{0,r-t\}$, consider the cardinality
$$\chi(r,u,t)=\left|\left\{ \bar a\in\left(\frac{\oinki}{\pi^{t-r+2u}\oinki}\right)^*\Bigg|\bar a^2=
\bar1, |a-1|=|\pi|^{t-r+u}\right\}\right|,$$
which we set as $1$ for $u=0$.
Then $$\stackrel{\rightarrow}e=n\left(2,1,1,\frac12\right)-n'\left(1,\frac12,\frac12,0\right)+\frac12\stackrel{\rightarrow}\chi,$$
 where $n$, $n'$, and $\stackrel{\rightarrow}\chi=(0,\chi_2,\chi_3,\chi_4)$ are as in Table 1. If $r=0$, then
$\stackrel{\rightarrow}e=(1,1,1,1)$.
\end{theorem}

\begin{table}
\scriptsize
\begin{tabular}{|c|c|c|c|c|c|}
\hline
$r$&$n$&$n'$&$\chi_2$&$\chi_3$&$\chi_4$\\
\hline\hline
$r=2h+1<2t$&$p^h$&$0$
&$0$&$\sum_{u=v}^{h}\chi(r,u,t)$&$\frac{\chi_3}2$\\ \hline
$r=2h<2t$&$p^h$&$(p-1)p^{h-1}$
&$\chi(r,h,t)$&$\sum_{u=v}^{h}\chi(r,u,t)$&$\frac{\chi_2+\chi_3}2$\\ \hline
$r=2t$&$p^t$&$(p-2)p^{t-1}$&
$\chi(r,t,t)$&$\chi_2$&$\chi_2$\\   \hline
$r>2t$&$2p^t$&$2(p-1)p^{t-1}$
&$0$&$0$&$0$\\
\hline
\end{tabular}
\normalsize
\caption{The invariants $n$, $n'$, and $\chi_i$ for the order $(\oinki\times\oinki)^{[t]}$.}

\end{table}

Our computations are greatly simplified by the explicit description of the branch of an order defined in \cite{eichler2},
see \S3. An additional simplification is obtain by translating our setting, back an forth, between two
known incarnations of the Bruhat-Tits tree. This is done in \S5. Although $e_1$ and $e_2$ were previously
computed in \cite[Corollary 1.6]{enumbers} and \cite[Theorem 1.8]{enumbers}, we noted a discrepancy with the values given there, where it is incorrectly stated
that there is a unique embedding when $r$ is small with respect to $t$.

Theorems 1.2-1.4 extend easily to the setting where $\Ea$ is an intersection of maximal orders.
Since embedding an order $\Ha$ into a full order $\Ea\subseteq\alge$  is equivalent to embedding 
$\Ha^{[t]}$ into $\Ea^{[t]}$ for any positive integer $t$, this case reduces easily to the above results,
as any full-rank intersection of maximal orders has the form $\Ea^{[t]}$ for an Eichler order $\Ea$
(cf. \cite[Th. 1.4]{eichler2}).

\section{Optimal representation fields}

In this section $\Pi$ denotes the set of all places in $K$, archimedean or otherwise. 
Let $\alge$ be an indefinite quaternion $K$-algebra, 
and let $\mathbb{O}=\mathrm{gen}(\Da)$ be a genus of orders of maximal rank 
in $\alge$, or as we say
in all that follows, full orders in $\alge$.

For any full order $\Da\subseteq\alge$, we define the adelization $\Da_{\ad}=\prod_{\wp\in\Pi}\Da_\wp$
endowed with the product topology, where by convention $\Da_\wp=\alge_\wp$  at archimedean places.
 The adelization $\alge_\ad$ is the set of all elements $a\in\prod_{\wp\in\Pi}\alge_\wp$ satisfying $a_\wp\in
\Da_\wp$ for almost all $\wp$, endowed with the only topology making every affine map of the form
$d\mapsto d+b$ an open embedding of
$\Da_\ad$ into $\alge_\ad$.  
Adelizations of other orders and algebras are defined analogously.
 In particular, we use the standard notations $\ad:=K_\ad$, $J_K:=\ad^*$, and $J_\alge:=\alge_\ad^*$, 
where $R^*$ is the group of units of the ring $R$.
We identify $K^*$ with a subset of $J_K$ via the 
diagonal embedding, and we let $N:J_\alge\rightarrow J_K$ denote the reduced norm. By abuse of notation,
the symbol $N$ is also used for the reduced norm on the global algebra $\alge$ or its localization $\alge_\wp$.
If $a\in\alge_{\ad}$, the order $\Da'=a\Da a^{-1}$ is defined by the local conditions $\Da'_\wp=a_\wp\Da_\wp a_\wp^{-1}$ at all finite places $\wp$.

 The spinor class field for the genus  $\mathbb{O}$ is the class field $\Sigma$ corresponding to the subgroup $K^*H(\mathbb{O})$ of the adele group $J_K$, where
$H(\mathbb{O})=H(\Da)=\{N(a)|a\Da a^{-1}=\Da\}$ is the spinor image. There exists a well defined distance map $\rho:\mathbb{O}\times \mathbb{O}\rightarrow\mathrm{Gal}(\Sigma/K)$, satisfying $\rho(\Da,a\Da a^{-1})=[N(a),\Sigma/K]$,
where $c\mapsto[c,\Sigma/K]$ is the Artin map on ideles \cite{cyclic}. When $\mathbb{O}=\mathbb{O}_I$ is a genus of Eichler orders, as in the introduction, the spinor class field $\Sigma$ is the largest exponent-$2$ sub-extension of the wide Hilbert class field for $K$ splitting at all finite places where the local valuation of $I$ is odd and at all infinite places splitting the algebra \cite[Theorem 1.2]{eichler2}.

 Assume $\Ha\subseteq\Da$ is optimal.  We define 
the representation field $F=F(\Da|\Ha)$,  the optimal representation field $F_{\mathrm{op}}=
F_{\mathrm{op}}(\Da|\Ha)$, and  the maximal representation
field $F_{\mathrm{max}}=F_{\mathrm{max}}(\alge|\Ha)$, as the class fields
 corresponding to the class groups $K^*H$,
$K^*H_{\mathrm{op}}$, and $K^*H_{\mathrm{max}}$, respectively, where
$$\begin{array}{lclcl}
H&=&H(\Da|\Ha)&=&\{N(a)|\Ha\subseteq a\Da a^{-1}\},\\
H_{\mathrm{op}}&=&H_{\mathrm{op}}(\Da|\Ha)&=&\{N(a)|\Ha\textnormal{ is optimal in }a\Da a^{-1}\},\\
 H_{\mathrm{max}} &=&H_{\mathrm{max}}(\alge|\Ha)&=&\{N(a)|a\Ha a^{-1}=\Ha\},\end{array}$$
 provided that they are subgroups of finite index in the adele group, otherwise we say that the corresponding
fields are undefined.

The order $\Ha$ embeds (respectively, embeds optimally) into a full order $\Da'\in\mathbb{O}$ if and only if  
$\rho(\Da,\Da')$ is the identity map on $F$  (resp. $F_{\mathrm{op}}$), when this field is defined. This is proved in
\cite{spinor} for $F$ and the proof for $F_{\mathrm{op}}$ is entirely analogous. The field $F_{\mathrm{max}}$,
if defined, is an upper bound for either $F$ or $F_{\mathrm{op}}$, that is independent of the genus.

 The existence of the representation field $F$ for orders in quaternion algebras follows from
 general results on representation of quadratic forms \cite{Chan}. For the particular orders
that concern us here,  this follows from the results in \cite{eichler2}. 
 When $\Da$ is an Eichler order, and $\Ha$ is an order in a semisimple commutative algebra $L$, then both $F_{\mathrm{op}}$ and $F_{\mathrm{max}}$ are defined as we prove below.
The field $F_{\mathrm{max}}$ is not defined when $\Ha$ spans an algebra $L\cong K[x]/(x^2)$,
 or has rank $3$, since the corresponding class group is not of finite index.
When $L=K\Ha\cong K\times K$, the set $H_{\mathrm{max}}$ contains the set of norms of the adelization
$L_\ad^*\cong J_K\times J_K$, where the norm is surjective. Next result follows:
\begin{proposition}\label{p21}
If $\Ha$ is an order spanning an algebra isomorphic to $K\times K$, then $F_{\mathrm{max}}(\alge|\Ha)$ is defined
and equals $K$. 
\end{proposition}
It was proved in \cite{continuity} that, when $L\cong K(\sqrt d)$
is a field, $F_{\mathrm{max}}(\alge|\Ha)$ is the largest field of the form $F(\Da|\Ha)$, where $\Da$ runs over the set of full orders in $\alge$ containing $\Ha$, and it can be computed  as follows:
$$F_{\mathrm{max}}(\alge|\Ha)=\left\{\begin{array}{ccl} L&\textnormal{ if }&\alge\cong\left(\frac{-1,d}K\right)\\ 
 K&&\textnormal{ otherwise }\end{array}\right..$$
\begin{proposition}
The field $F_{\mathrm{op}}=F_{\mathrm{op}}(\Ea|\Ha)$ is defined whenever $\Ea$ is an Eichler order.
\end{proposition}
\begin{proof}
Define the local component
$$H_{\mathrm{op},\wp}=
H_{\mathrm{op},\wp}(\Da|\Ha)=\{N(a)|a\in\alge_\wp,\, \Ha_\wp\textnormal{ is optimal in }a\Da_\wp a^{-1}\},$$
while $H_\wp=H_\wp(\Da)$ is defined analogously. 
In fact, $H_\wp$ is always 
a group, since it is the image, under the norm, of a stabilizer.
If $\Da$ is an Eichler order, the group $H_\wp$ contains $\oink_\wp^*K_\wp^{*2}$, where $\oink_\wp^*$
is the group of local units. Since $[K_\wp^*:\oink_\wp^*K_\wp^{*2}]=2$ for every finite place, and 
$H_{\mathrm{op},\wp}H_\wp=H_{\mathrm{op},\wp}$,
 it follows that $H_{\mathrm{op,\wp}}$ is either $K_\wp^*$ or $\oink_\wp^*K_\wp^{*2}$,
whence it is a group containing $H_\wp$, and the field $F_{\mathrm{op}}\subseteq\Sigma$ is defined.
\end{proof}

\begin{remark}
All results in this section hold if $K$ is just assumed to be a global field and $\infty$ is replaced by a finite
set $S\subseteq\Pi$ containing the archimedean places, if any, and at least one place in $S$ splits $\alge$ \cite{abelianos}.
\end{remark}
  
\section{Local computations and trees}
 
Let $k$ be a local field with ring of integers $\oinki$. Recall that there is a one-to-one correspondence between maximal orders in $\matrici_2(k)$, or equivalently homothety classes of full rank lattices in $k^2$, and
vertices in the Bruhat-Tits tree for $PSL_2(k)$ \cite{trees}. We call it the  BT-tree $\mathfrak{T}=\mathfrak{T}(k)$ 
in all that follows, and let $\delta$ denote the usual distance on the graph $\mathfrak{T}$.
  Let $\Ha$ be an order in $\matrici_2(k)$, and let 
$S_0(\Ha)$ be the set of vertices corresponding to the maximal orders containing $\Ha$. Let $\mathfrak{S}_0(\Ha)$ be the
branch of $\Ha$, i.e., the largest subgraph of the BT-tree whose vertices are in $S_0(\Ha)$.
 It was proved in \cite{eichler2} that branches fall in a rather small family, they are usually the maximal subtree with 
vertices lying no farther than a fixed distance $p=p(\Ha)$, the depth of the order, from a path (Figure 1A), which we call the stem of the order. The stem can be an infinite ray (Figure 1B), a maximal path (Figure 1C), or be reduced to a point (Figure 1D).  Any of these sets is called a $p$-thick line. When the stem is reduced to a point,  $\mathfrak{S}_0(\Ha)$ is a ball of radius $p$. Vertices in the stem are called stem vertices, while non-stem vertices are called leaf vertices. If the stem is a ray or a non-trivial finite path, its endpoints are called stem borders. A thick line is a $p$-thick line for some $p$. The only order $\Ha$ for which $\mathfrak{S}_0(\Ha)$ is not a thick line, other that the trivial order $\oinki\mathfrak{1}$,
is the rank 2 order $\Ha=\oinki[c]$ where $c^2=0$. In this case $\mathfrak{S}_0(\Ha)$ is a graph
with only leaf vertices, called the infinite leaf (Figure 1E).
It can be obtained as a infinte union of balls of increasing radius with a common endpoint and centers lying on a ray \cite[Prop. 4.4]{eichler2}. The branch of an Eichler order is a finite path. A finite path can be described by a walk $v_0v_1v_2\dots v_n$ without backtracking, i.e., 
a sequence of vertices in $\mathfrak{G}$ without repetitions, where
any two consecutive vertices are neighbors. Similar conventions apply to rays or maximal paths.
The depth $p(v)$
of a vertex $v$ in a branch $\mathfrak{S}$,  is the radius of the largest ball with center
$v$ contained in $\mathfrak{S}$.  
It follows easily from the explicit description of branches that every path in a branch whose associated walk 
$v_0v_1v_2\dots v_n$ satisfy $p(v_0)=r$ and $p(v_1)=r-1$, must also satisfy  $p(v_k)=r-k$ for any $k=2,\dots,n$.
In particular $n\leq r$. Such a path is said to go outwards through the leaves. See \cite{eichler2} for details.

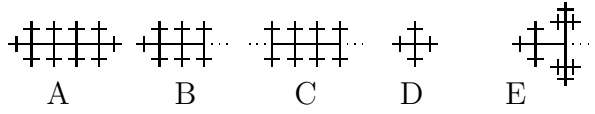
\begin{figure}\qquad
\unitlength 1mm 
\linethickness{0.4pt}
\ifx\plotpoint\undefined\newsavebox{\plotpoint}\fi 
\begin{picture}(90,13)(0,0)
\put(0,8){\line(1,0){15}}
\put(1,7){\line(0,1){2}}
\put(3,5){\line(0,1){6}}\put(2,10){\line(1,0){2}}\put(2,6){\line(1,0){2}}
\put(6,5){\line(0,1){6}}\put(5,10){\line(1,0){2}}\put(5,6){\line(1,0){2}}
\put(9,5){\line(0,1){6}}\put(8,10){\line(1,0){2}}\put(8,6){\line(1,0){2}}
\put(12,5){\line(0,1){6}}\put(11,10){\line(1,0){2}}\put(11,6){\line(1,0){2}}
\put(14,7){\line(0,1){2}}
\put(17,8){\line(1,0){9}}
\put(18,7){\line(0,1){2}}
\put(20,5){\line(0,1){6}}\put(19,10){\line(1,0){2}}\put(19,6){\line(1,0){2}}
\put(23,5){\line(0,1){6}}\put(22,10){\line(1,0){2}}\put(22,6){\line(1,0){2}}
\put(26,5){\line(0,1){6}}\put(25,10){\line(1,0){2}}\put(25,6){\line(1,0){2}}
\multiput(26,8)(1,0){4}{\line(1,0){.0625}}
\put(35,8){\line(1,0){9}}
\put(35,5){\line(0,1){6}}\put(34,10){\line(1,0){2}}\put(34,6){\line(1,0){2}}
\put(38,5){\line(0,1){6}}\put(37,10){\line(1,0){2}}\put(37,6){\line(1,0){2}}
\put(41,5){\line(0,1){6}}\put(40,10){\line(1,0){2}}\put(40,6){\line(1,0){2}}
\put(44,5){\line(0,1){6}}\put(43,10){\line(1,0){2}}\put(43,6){\line(1,0){2}}
\multiput(44,8)(1,0){4}{\line(1,0){.0625}}\multiput(35,8)(-1,0){4}{\line(1,0){.0625}}
\put(51,8){\line(1,0){6}}
\put(52,7){\line(0,1){2}}
\put(54,5){\line(0,1){6}}\put(53,10){\line(1,0){2}}\put(53,6){\line(1,0){2}}
\put(56,7){\line(0,1){2}}
\put(67,8){\line(1,0){7}}
\put(68,7){\line(0,1){2}}
\put(70,5){\line(0,1){6}}\put(69,10){\line(1,0){2}}\put(69,6){\line(1,0){2}}
\put(74,2.5){\line(0,1){11}}\put(72,11){\line(1,0){4}}\put(72,5){\line(1,0){4}}
\put(73,10){\line(0,1){2}}\put(75,10){\line(0,1){2}}
\put(73,4){\line(0,1){2}}\put(75,4){\line(0,1){2}}
\put(73,3.4){\line(1,0){2}}\put(73,12.6){\line(1,0){2}}
\multiput(75,8)(1,0){4}{\line(1,0){.0625}}
\put(5,0){A}\put(22,0){B}\put(38,0){C}\put(52,0){D}\put(66,0){E}
\end{picture}
\caption{Some branches of orders when $k=\mathbb{Q}_3$.} 
\end{figure}

 Two rays in a graph $\mathfrak{G}\subseteq\mathfrak{T}$ are equivalent if the 
corresponding walks  $v_0v_1v_2\dots$ 
and $v'_0v'_1v'_2\dots$ satisfy $v'_t=v_{t+s}$ for $s\in \mathbb{Z}$ fixed an every $t$ large enough.
An equivalence class of rays is called an end of $\mathfrak{G}$.
 Every end of 
$\mathfrak{G}$ corresponds to a unique end in $\mathfrak{T}$, and in this sense we say that $\mathfrak{G}$ contains
an end or a certain end $e$ is contained in $\mathfrak{G}$ (written $e\in\mathfrak{G}$). It is not
hard to see that each thick ray and each infinite leaf contains a unique end, while a thick maximal path contains two ends. We also make extensive use of the following result (cf. \cite[Prop. 2.4]{eichler2}):

\begin{proposition}\label{p31}
For any order $\Ha$, and any positive integer $t$, we have $$S_0\left(\Ha^{[t]}\right)=\Big\{
v\in V(\mathfrak{T})\Big| \delta(v,v')\leq t\textnormal{ for some }v'\in S_0(\Ha)\Big\}.$$
\end{proposition}

From the shapes of thick paths or infinite leaves, next result is straightforward:

\begin{proposition}\label{parity}
For any thick line $\mathfrak{S}$ with a set $T$ of stem vertices,
we denote by $c_T(w,w')$ the number of points in the stem $T$ lying in the path
joining $w$ and $w'$. Then, for every pair $v$ and $v'$ of endpoints
$$\delta(v,v')\equiv\left\{\begin{array}{ccl} c_T(v,v')-1&\textnormal{ if }&c_T(v,v')\neq0\\ 
 0&&\textnormal{ otherwise }\end{array}\right.\qquad(\textnormal{mod }2).$$
In particular, the distance between every pair $v$ and $v'$ of endpoints in either, a ball or an infinite leaf,  is even.
\end{proposition}

Let $\Da$ be a full order in the local quaternion algebra $\alge$, and let  $L$ be a subalgebra of
$\alge$. We note that $L\cap\Da$ can be characterized as the largest order in $L$ contained
in $\Da$. In particular, If $\Ea$ is an Eichler order, $L\cap\Ea$ is the largest order $\Ha$ in $L$
such that $\mathfrak{S}_0(\Ha)$ contains the finite path $\mathfrak{S}_0(\Ea)$. We conclude
that an order in $L$ is optimal in $\Ea$ if and only if it satisfies the following two conditions:
\begin{enumerate}
\item $\mathfrak{S}=\mathfrak{S}_0(\Ha)$ is minimal among the branches of orders in $L$ containing the path
$\mathfrak{S}_0(\Ea)$.
\item $\Ha$ is maximal among the suborders of $L$ having the branch $\mathfrak{S}$.
\end{enumerate}
We show that, when $L$ is a proper subalgebra of $\alge$, a full order in $L$ is completely determined by its
branch. In fact, if $L$ is a semisimple conmutative algebra, then every order in $L$ has the form $\oink_L^{[t]}$
for some non-negative integer $t$ \cite[Lemma 4.1]{eichler2}, and their branches are certainly different because of
Proposition \ref{p31}. A similar phenomenon occurs when $L$ is generated by a nilpotent element since
for any pair of full orders $\Ha_1$ and $\Ha_2$ in $L$ we have either $\Ha^{[t]}_1=\Ha_2$ or 
$\Ha^{[t]}_2=\Ha_1$ for some $t$. In fact, we can assume $L=k\left[\lbmatrix 0100\right]$ and any order in $L$
has the form $\Ha=\oinki\left[\lbmatrix 0{\pi^r}00\right]$ for some $r\in\mathbb{Z}$.
The statement follows, therefore, from the following result:

\begin{lemma}\label{ranktre}
Every order $\Ha$ in the three dimensional algebra $L=\lbmatrix kk0k$ is the intersection of the maximal orders in a thick ray. The algebra $L$ is uniquely determined by the end of the ray and conversely. 
\end{lemma}

\begin{proof}
Let $\oinki c=\Ha\cap\lbmatrix 0k00$. Then $\Ha$ has a basis of the form $\{c,\mathfrak{1},a\}$, where 
$\mathfrak{1}$ denotes 
the identity matrix, and $a$ is not in the space spanned by $\mathfrak{1}$ and $c$, whence  it has different eigenvalues.
By an appropriate choice of basis we can assume $c=\lbmatrix 0{\pi^r}00$ and $a=\lbmatrix x00y$, where
$x\equiv y\mod \pi^r$ and the result follows from \cite[Theorem 1.2]{eichler2}. For the last statement, we
observe that two orders $\Ha$ and $\Ha_0$ span the same algebra if and only if for some $t\geq0$ we have
$\Ha^{[2t]}\subseteq\Ha_0^{[t]}\subseteq\Ha$, and the corresponding condition for trees
characterize thick rays with the same end.
\end{proof}

In what follows, if $S$ is an $r$-thick path, we denote by $u(v)$ the steam vertex that is closest to $v$.
The simplified graph associated to an $r$-thick path is the graph obtained by identifying two
leaf vertices $v$ and $v'$ at the same distance from the stem, whenever $u(v)=u(v')$, 
and identifying two edges whenever its corresponding endpoints 
are identified or coincide (see Figure 6 in \S6). The simplified
graph of a leaf can be defined similarly by identifying all vertices at the same depth. The latter graph is a ray.
  Similarly, for every subtree
$W$ containing the stem we can define $u_W(v)$ as the vertex of  $W$ that is closest to $v$. The tree $W$ is
usually obtained adding,  to the stem, one or more paths going outwards through the leaves. Then the simplified
graph with these paths expanded can be defined by replacing $u$ by $u_W$ and the stem by $W$ in the 
preceding definition (see Figure 2).

\begin{figure}
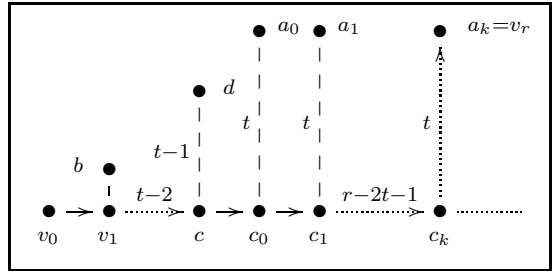

\[ \fbox{ \xygraph{
!{<0cm,0cm>;<.8cm,0cm>:<0cm,.8cm>::} 
!{(1,1) }*+{\bullet}="v0" !{(1,0.5) }*+{{}^{v_0}}="v0n"
!{(2,1) }*+{\bullet}="v1" !{(2,0.5) }*+{{}^{v_1}}="v0n"
!{(3.5,1) }*+{\bullet}="c" !{(3.5,0.5) }*+{{}^{c}}="cn"
!{(4.5,1) }*+{\bullet}="c0" !{(4.5,0.5) }*+{{}^{c_0}}="c0n"
!{(5.5,1) }*+{\bullet}="c1" !{(5.5,0.5) }*+{{}^{c_1}}="c1n"
!{(7.5,1) }*+{\bullet}="ck" !{(7.5,0.5) }*+{{}^{c_k}}="ckn"
!{(2,1.7) }*+{\bullet}="b" !{(1.5,1.7) }*+{{}^{b}}="bn"
!{(3.5,3) }*+{\bullet}="d" !{(4,3) }*+{{}^{d}}="bn"
!{(4.5,4) }*+{\bullet}="a0" !{(5,4) }*+{{}^{a_0}}="a0n"
!{(5.5,4) }*+{\bullet}="a1" !{(6,4) }*+{{}^{a_1}}="a1n"
!{(7.5,4) }*+{\bullet}="ak" !{(8.5,4) }*+{{}^{a_k=v_r}}="akn"
!{(9,1) }*+{}="e"
"v0"-@{->}"v1" "v1"-"b" "v1"-@{.>}^{t-2}"c" "c"-@{--}^{t-1}"d" "c"-@{->}"c0" "c0"-@{--}^t"a0"
"c0"-@{->}"c1" "c1"-@{--}^t"a1" "c1"-@{.>}^{r-2t-1}"ck" "ck"-@{.>}^t"ak" "ck"-@{.}"e" 
 } }
\]
\caption{The simplified graph of a thick ray with one expanded path, namely the one joining $c_0$ and $v_0$.} 
\end{figure}

\begin{lemma}\label{newlemma}
Let now $R$ be a $t$-thick ray with an end $e$, and let $P\subseteq R$ be a path of length r whose 
vertices in order are $v_0,v_1,\dots,v_r$. Then $R$ is the smallest thick ray with end $e$ containing $P$ if and only if
the following conditions hold:
\begin{enumerate} \item $v_0$ and $v_r$ are endpoints of $R$,
\item either $u(v_0)$ or $u(v_r)$ is the stem border of $R$, 
\item $r\geq 2t$.
\end{enumerate}
\end{lemma}

\begin{proof}
Let $c_0c_1,\dots$ be the walk corresponding to the stem $T$ of $R$ as in Figure 2. If $P$ does not have an endpoint, say $v_0$, which is an endpoint of $R$ and $u(v_0)=c_0$, then $P$ is contained in an smaller thick ray, namely the $t$-thick path whose stem is the ray in $e$ corresponding to the walk $c_1c_2c_3\dots$, which we usually call the ray joining $c_1$ and $e$.  Assume next that $v_0$ is as above, while either $r<2t$ or $\delta(v_r,T)<t$. 
In either case $P$ is contained in the $(t-1)$-thick ray whose stem joins $e$ and the neighbor $c$ of $c_0$ lying between
$c_0$ and $v_0$ (see Figure 2).  Conversely, if (1)-(3) are satisfied, any ray in $e$ containing $P$ must have thickness $t'\geq t$, and  $c_{t'-t}$ must be a point in the stem. The result follows.
\end{proof}

\subparagraph{Proof of Theorem \ref{t1}}

The  existence of an optimal embedding of an order $\Ha$ of rank $2$ into some order in 
$\mathbb{O}_I$ can be proved locally. A local embedding $\phi:\Ha\rightarrow\Ea$ is optimal if and only if
$S_0(\Ea)$ contains an endpoint of $\mathfrak{S}_0\Big(\phi(\Ha)\Big)$, 
and it is easy to see that any branch containing paths
of length $\alpha_\wp$, contains one such path starting from an endpoint. Any such path
has the form $\mathfrak{S}_0(\Ea')$ for a local Eichler order of level $\alpha_\wp$.  Recall that
 $$H_\wp(\Ea)=\left\{\begin{array}{cl}
N(\alge_\wp^*)&\textnormal{ at infinite places }\wp,\\
\oink_\wp^*K_\wp^{*2}&\textnormal{ if }\alpha_\wp\textnormal{ is even,}\\ K_\wp^*&\textnormal{otherwise.} 
\end{array}\right.$$

Asume first $L\cong K(\sqrt d)$ is a field.
It follows from the contention $F_{\mathrm{op}}\subseteq F_{\mathrm{max}}\subseteq  L$
(\S2), that $\Ha$ embeds optimally into
 every order in the genus, unless $F_{\mathrm{op}}= L$. 
If this condition is satisfied, then necessarily $L\subseteq\Sigma(\Da)$, which impplies (b) and (c) by the preceding 
formulas, and $F_{\mathrm{max}}=L$, which implies (a). On the other hand, if all three conditions
are satisfied, the maximal order $\oink_L$ is selective for the genus $\mathbb{O}_0$ of maximal orders
 \cite{FriedmannQ}. In particular, $L$ embeds into $\alge$ and $F(\Da|\oink_L)=L$ for 
any maximal order $\Da$ containing $\oink_L$.
There are no ramified places for $L/K$, and for any finite place $\wp$, inert for $L/K$, 
 the branch of $\Ha=\oink_L^{[t]}$ is a ball, whence the distance from every pair of its vertices at depth $0$ is even
by Proposition \ref{parity}. We conclude, from  the preceding characterization of optimal embeddings, 
 that $H_{\mathrm{op},\wp}(\Ea|\Ha)=\oink_\wp^*K_\wp^*$, when $\alpha_\wp$ is even and $\wp$ inert. Since the
equality $H_{\mathrm{op},\wp}(\Ea|\Ha)=N(L_\wp^*)$  is trivial at infinite places by (c), the result follows. The case  $L\cong K\times K$ is similar, since in this case $F_{\mathrm{max}}(\alge|\Ha)=K$ by Proposition \ref{p21}. 

Assume next $L\cong K[x]/(x^2)$.
It suffices to prove that in this case $F_{\mathrm{op}}(\Da|\Ha)$ is the spinor class field $\Sigma(\Da)$. This follows
if $H_{\mathrm{op},\wp}(\Da|\Ha)=H_\wp(\Da)$ holds at all places $\wp$. This is immediate at
infinite places, so we assume $\wp$ is finite. One contention is immediate, while the other follows from Proposition \ref{parity}. When $L$ is three-dimensional, the result follows, since for any optimal embedding of a local order $\Ha_\wp=\oink_\wp[a,c]$ of rank $3$, where $a$ is semisimple and $c$ nilpotent, the induced embedding 
of $\oink_\wp[c]$ is also optimal by Lemma \ref{newlemma}, since  the vertex 
 $v_0$ in Figure 2 must be an endpoint of the infinite leaf $\mathfrak{S}_0(\oink_\wp[c])$.
\qed

\section{Two realizations of the Bruhat-Tits tree}

 In any metric space we denote by $B_z[r]$ the closed ball of centre $z$ and 
radius $r$. In a local field $k$ with absolute value $\rho$ and uniformizing parameter $\pi$,
we also write $B_z^{[r]}$ instead of $B_z[\rho(\pi)^r]$.
 Since in $k$ any element of a ball is its center, for every pair of balls $B$ and $D$, either 
$B\cap D=\emptyset$ or one ball is contained in the other. In the latter case,
there is an element $z$ of $k$ and two integers $r,s$ such that $B=B_{z}^{[r]}$ and 
 $D=B_{z}^{[s]}$, hence we can define the distance between them by $d(D,B)=|r-s|$.
Furthermore, if $B$ and $D$ are disjoint, and if $C$ is the smallest ball containing both, we define
$d(B,D)=d(B,C)+d(C,D)$. With this distance the set of balls in $k$ is a metric space.
 We define a graph $\mathfrak{G}$ whose vertices are the balls and 
there is an edge between two balls if and only if the distance between them is $1$.
It is easy to see that this graph is a tree. Figure 1 shows part of the graph for $k=\mathbb{Q}_2$.  

\begin{figure}
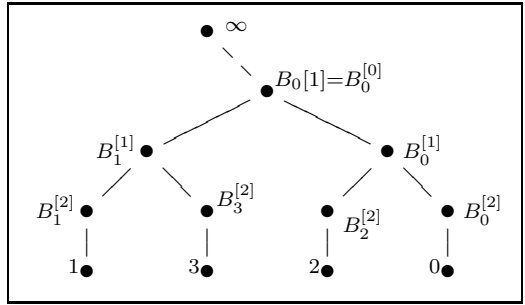

\[ \fbox{ \xygraph{
!{<0cm,0cm>;<.8cm,0cm>:<0cm,.8cm>::} 
!{(11,19) }*+{\bullet}="i" !{(11.5,19) }*+{{}^\infty}="i1"
!{(12,18) }*+{\bullet}="b0" !{(12.7,18.2) }*+{\ \ \ \ {}^{B_0[1]=B_0^{[0]}}}="b0n"
!{(10,17) }*+{\bullet}="b11" !{(9.5,17) }*+{{}^{B_1^{[1]}}}="b11n"
!{(14,17) }*+{\bullet}="b10" !{(14.6,17) }*+{{}^{B_0^{[1]}}}="b10n"
!{(9,16) }*+{\bullet}="b21" !{(8.5,16) }*+{{}^{B_1^{[2]}}}="b21n"
!{(11,16) }*+{\bullet}="b23" !{(11.5,16.2) }*+{{}^{B_3^{[2]}}}="b23n"
!{(13,16) }*+{\bullet}="b22" !{(13.6,15.8) }*+{{}_{B_2^{[2]}}}="b22n"
!{(15,16) }*+{\bullet}="b20" !{(15.6,16) }*+{{}^{B_0^{[2]}}}="b20n"
!{(9,15) }*+{\bullet}="d1" !{(8.8,15) }*+{{}^1}="d1n"
!{(11,15) }*+{\bullet}="d3" !{(10.8,15) }*+{{}^3}="d3n"
!{(13,15) }*+{\bullet}="d2" !{(12.8,15) }*+{{}^2}="d2n"
!{(15,15) }*+{\bullet}="d0" !{(14.8,15) }*+{{}^0}="d0n"
"i"-@{--}"b0" "b0"-"b11" "b0"-"b10" "b10"-"b22" "b10"-"b20" "b11"-"b21" "b11"-"b23"
"b20"-@{--}"d0" "b21"-@{--}"d1" "b22"-@{--}"d2" "b23"-@{--}"d3"
 } }
\]
\caption{$\mathfrak{T}(\mathbb{Q}_2)$ seen as the graph of balls $\mathfrak{G}$.
Dashed lines represent the class of rays that correspond to the element of $\mathbb{P}_1(\mathbb{Q}_2)$ that 
appears in its 
end.} \label{figarb}
\end{figure}

We give an explicit isomorphism between $\mathfrak{G}$ and the BT-tree  $\mathfrak{T}$ defined in \S3. 
Remember that if $\Lambda$ is a lattice of rank two in 
$k^2$, then $\mathfrak{D}_{\Lambda}=\{T \in \matrici_2(k) \mid T \Lambda \subseteq \Lambda \}$, is a maximal order in $\matrici_2(k)$. 
Furthermore every maximal order in $\matrici_2(k)$ is $\mathfrak{D}_{\Lambda}$ for some lattice 
$\Lambda\subseteq k^*$ of rank two.  
For any element $z \in k$ and any $n \in \mathbb{Z}$  we denote by $ \Lambda_{z,n}$ the lattice generated by $(1,z)$ 
and 
$(0,\pi^n)$.   It is apparent  that if $x \equiv y $ mod $(\pi^n),$
then $\mathfrak{D}_{\Lambda_{x,n}}= \mathfrak{D}_{\Lambda_{y,n}}$. 
Hence, the map $B_z^{[n]}  \mapsto \mathfrak{D}_{\Lambda_{z,n}}$, from the set of balls to the set
of maximal orders, is well defined. 
This map is bijective since the lattices $\Lambda_{z,n}$ are exactly the lattices whose projection to the first 
coordinate is the ring $\oinki$, whence every lattice is a multiple of exactly one lattice of the form
$\Lambda_{z,n}$, for $(z,n) \in k \times \mathbb{Z}$. It is straightforward that this correspondence preserves 
neighbors, and so does its inverse since the valencies of the vertices in either tree is the same.

By a descending  ray, we mean the ray defined by a walk $B_0B_1B_2\dots$, as before, 
satisfying $B_i\supset B_{i+1}$. An ascending ray is defined analogously. 
Since $k$ is a complete metric space, for every 
descending ray as above there is a unique element $z \in \bigcap_{i=1}^{\infty} B_i \subseteq k$.
Moreover $\{z\}=\bigcap_{i=1}^{\infty} B_z^{[i]}$, whence the function 
$\psi\left[(B_i)_{i=1}^{\infty}\right]=z$ that associates, to each descending ray, its intersection, is surjective. 
 It is not hard to show that every ray in the graph is ascending or equivalent to a
descending ray, whence we can identify the set of ends in the tree with 
 the projective line $\mathbb{P}_1(k)$,
by sending the class of ascending rays to $\infty$ (Figure \ref{figarb}).

Let us denote by $\mathbb{M}$ the group of M\"{o}bius tranformations on $k$. Remember that  
$\mathbb{M} \cong \mathrm{PGL}_2(k)=\mathrm{GL}_2(k)/k^*$ and  it acts on the projective line $\mathbb{P}_1(k)$. 
We would like to define an action of $\mathbb{M}$ on $\mathfrak{G}$ but the image of a ball in $k$ is not always 
a ball in  $k$, so we use balls in $\mathbb{P}_1(k)$ instead.  A ball in $\mathbb{P}_1(k)$ is either 
a ball in $k$ or the complement in $\mathbb{P}_1(k)$ of a ball in $k$. 
 It is known that $\mathbb{M}$ acts on the set of balls in $\mathbb{P}_1(k)$.

Let $B$ be a ball in $k$. The partition of $\mathbb{P}_1(k)$ defined by  $B$ is the collection 
$\mathfrak{P}(B)=\{B^c,B_1, \dots , B_p\}$ where $B_1, \dots , B_p$ are all the balls in $k$ contained in and 
adjacent to $B$. 
It is clear that $\mathfrak{P}(B)$ is a partition of $\mathbb{P}_1(k)$, and $B^c$ is the set that contains the point $
\infty$. It is not hard to see that every partition $\mathfrak{P}$ of $\mathbb{P}_1(k)$ into $p+1$ balls has this 
form, i.e.  $\mathfrak{P}=\mathfrak{P}(E)$ for some ball $E$ in $k$.
Since $\sigma \in \mathbb{M}$ is a bijection in $\mathbb{P}_1(k)$, the set 
$\sigma(\mathfrak{P}(B))=\{\sigma(D) \mid D \in \mathfrak{P}(B)\}$ is 
also a partition.  It  follows that we can define an action of $\mathbb{M}$ 
in the set of balls in $k$  by setting $\sigma * B =E$ whenever $\sigma (\mathfrak{P}(B)) = \mathfrak{P}(E)$.  

For an invertible matrix $u=\lbmatrix abcd$ we call $\tilde{u}$ the M\"{o}bius 
transformation 
given by $\tilde{u}(z)=\frac{az+b}{cz+d}$. An easy but extended computation in the generators of $\mathbb{M}$, 
proves that the 
action of $\mathbb M$ on the vertices of the Bruhat-Tits tree defined by partitions coincides with the action 
by conjugation on the vertices seen as maximal orders, i.e. for every ball 
$B=B_{z}^{[r]}$,
we have  $\tilde{a} *B= B_{z'}^{[r']}$ if and only if $
 \mathfrak{D}_{\Lambda_{z',r'}} =a \mathfrak{D}_{\Lambda_{z,r}} a^{-1}$. 
Finally, since the latter action preserve the edges of the tree, the 
former one does. Summarizing, we can see the natural action of the group of M\"{o}bius transformations on the 
Bruhat-Tits tree as
an extension of its action on the projective line $\mathbb{P}_1(k)$. We use this 
action to study branches of the tree in all that follows.  

\begin{example}
Let $\Ha=\oinki\left[\lbmatrix 0100\right]$ be the order generated by a nilpotent element. Since the normalizer of a maximal
order $\Da$ is $k^*\Da^*$, the branch $S_0(\Ha)$ is the set of orders that are invariant under conjugation by the unit
$h=\lbmatrix 1101$, which generates $\Ha$. 
This element corresponds to the Moebius transformation $z\mapsto z+1$. Note that for a  
Moebius transformation $\tau$ fixing $\infty$, the ball $\tau*B$ is simply the image $\tau(B)$. We conclude that
$S_0(\Ha)$ is the set of maximal orders corresponding to balls $B$ satisfying $B=B+1$, i.e., balls of radius $1$ or
larger. This collection forms an infinite leaf as in Figure 1E, whose endpoints are the balls of radius $1$.
We recover thus \cite[Prop. 4.4]{eichler2}.
\end{example}

\section{Cross ratio as an invariant in the Bruhat-Tits tree}

The bijection described in \S4
allow us to see the elements of $\mathbb{P}_1(k)$ as ends of the tree.
 As usual, the cross ratio is defined by 
\begin{equation}
 [a,b;c,d]= \frac{a-c}{b-c}  \cdot \frac{b-d}{a-d}\in k,  \label{frazoncruzada}
\end{equation}
for any quartet $(a,b,c,d)\in\mathbb{P}_1(k)^4$ without repeated coordinates, and with the usual 
conventions regarding the value $\infty$. 
The cross ratio is invariant under the action of $\mathbb{M}$ described in \S4. Concretely, for every 
$a,b,c,d$ in $\mathbb{P}_1(k)$, and every $\sigma \in \mathbb{M}$, we have 
$[a,b;c,d]=[\sigma(a),\sigma(b) ; \sigma(c),\sigma(d)]$. 

For an $n$-tuple $S=(B_1, \dots , B_n)$ of vertices, we define the hull of $S$, as the pair  
$(\mathfrak{h},S)$ where $\mathfrak{h}$ is the minimal connected subgraph of the Bruhatt-Titts tree containing 
$\{B_1, \dots , B_n\}$. Consider a pair of $n$-tuples $S=(B_1, \dots , B_n)$, $S'=(B_1', \dots , B_n')$,
while $(\mathfrak{h},S)$ and $(\mathfrak{h}',S')$ denote the corresponding hulls. We say that $(\mathfrak{h},S)$ and $(\mathfrak{h}',S')$ are isomorphic if 
$B_i \mapsto B_i'$ extends to an ismorphism of graphs  $\phi:\mathfrak{h}\rightarrow\mathfrak{h}'$.
We say that $S$ and $S'$ are conjugated if there is an element $\sigma \in \mathbb{M}$ such that
$\sigma * B_i=B'_i$ for every $i \in \{1,  \dots ,n\}$. 
We write $\sigma *S=S'$ in what follows, while notations like $\sigma(a,b,c)$
 for $a,b,c\in\mathbb{P}^1(k)$ must be interpreted similarly.
Note that conjugated $n$-tuples have isomorphic hulls.

\begin{figure}
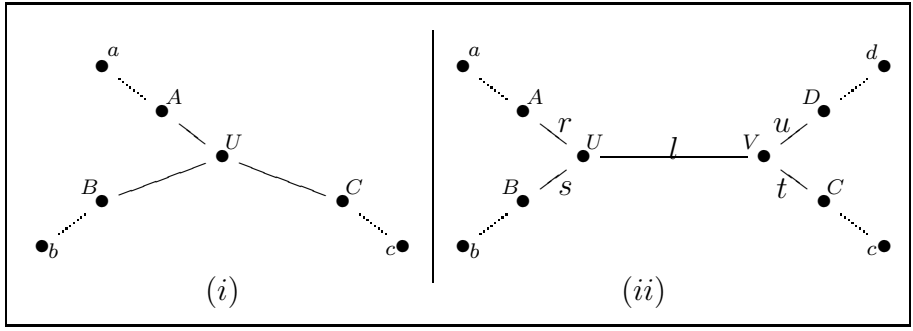

\[ \fbox{ \xygraph{
!{<0cm,0cm>;<.8cm,0cm>:<0cm,.6cm>::} 
!{(7,17) }*+{\bullet}="a"   !{(7.2,17.2) }*+{{}^a}="an"
!{(8,16) }*+{\bullet}="A"   !{(8.2,16.2) }*+{{}^{A}}="An"
!{(9,15) }*+{\bullet}="U"   !{(9.2,15.2) }*+{{}^{U}}="Un"
!{(7,14) }*+{\bullet}="B"   !{(6.8,14.2) }*+{{}^{B}}="Cn"
!{(11,14) }*+{\bullet}="C"  !{(11.2,14.2) }*+{{}^{C}}="Dn"
!{(6,13) }*+{\bullet}="b"   !{(6.2,12.8) }*+{{}^b}="bn"
!{(12,13) }*+{\bullet}="c"  !{(11.8,12.8) }*+{{}^c}="cn"
!{(12.5,18) }*+{}="down" 
!{(9,12) }*+{(i)}="i" 
!{(16,12) }*+{(ii)}="ii" 
!{(12.5,12) }*+{}="up" 
!{(16.5,15.2)}*+{l}="l1"
!{(14.7,15.7)}*+{r}="r1"
!{(14.7,14.3)}*+{s}="s1"
!{(18.3,14.3)}*+{t}="t1"
!{(18.3,15.7)}*+{u}="u1"
!{(13,17) }*+{\bullet}="al"   !{(13.2,17.2) }*+{{}^{a}}="aln"
!{(14,16) }*+{\bullet}="Al"   !{(14.2,16.2) }*+{{}^{A}}="Aln"
!{(13,13) }*+{\bullet}="bl"   !{(13.2,12.8) }*+{{}^{b}}="bln"
!{(14,14) }*+{\bullet}="Bl"   !{(13.8,14.2) }*+{{}^{B}}="Bln"
!{(15,15) }*+{\bullet}="Ul"   !{(15.2,15.2) }*+{{}^{U}}="Uln"
!{(18,15) }*+{\bullet}="Vl"  !{(17.8,15.2) }*+{{}^{V}}="Vln"
!{(20,17) }*+{\bullet}="dl"  !{(19.8,17.2) }*+{{}^{d}}="dln"
!{(19,16) }*+{\bullet}="Dl"  !{(18.8,16.2) }*+{{}^{D}}="Dln"
!{(20,13) }*+{\bullet}="cl"  !{(19.8,12.8) }*+{{}^{c}}="cln"
!{(19,14) }*+{\bullet}="Cl"  !{(19.2,14.2) }*+{{}^{C}}="Cln"
"a"-@{.}"A" "A"-"U" "U"-"B" "U"-"C" "B"-@{.}"b" "C"-@{.}"c"
"Al"-"Ul" "Ul"-"Bl" "Ul"-"Vl" "Vl"-"Dl" "Vl"-"Cl"
"al"-@{.}"Al" "bl"-@{.}"Bl" "cl"-@{.}"Cl" "dl"-@{.}"Dl"
"down"- "up"
} }
\]
\caption{ The hulls in Proposition 5.1 and 5.2. Some continuous lines might be reduced to points, e.g., 
$A=U$ in (i), $U=V$ or $D=V$ in (ii) are possible variations.} 
\label{figmob}
\end{figure}

\begin{proposition}\label{triplets}
Let $k$ a local field. Let $S=(A,B,C)$ and $S'=(A',B',C')$ two triplets of balls with isomorphic hulls. 
Then $S$ and $S'$ are conjugated. 
\label{thmob}
\end{proposition}

\begin{proof}
Let $a,b,c\in\mathbb{P}_1(k)$ be ends beyond $A,B,C$ and let $U$ be the central vertex of the hull,
as in Figure $\textnormal{\ref{figmob}}(i)$. If, for instance, $A$ lies between $B$ and $C$, we can choose $U=A$
and choose $a$ in another direction, so the hull still looks like Figure $\textnormal{\ref{figmob}}(i)$
 with a line reduced to a point. Let
$a',b',c'$ and $U'$ be the corresponding elements for the hull of $S'$.     
Let $\sigma$ be the unique M\"obius transformation such that $\sigma(a,b,c)=(a',b',c')$.
In figure $\textnormal{\ref{figmob}}(i)$ the ball $U$ is uniquely determined by $a,b,c$. Concretely it is the 
only ball whose partition separates every two elements in $\{a,b,c\}$. Since $\sigma * U$ is the 
only ball whose partition separates every two elements in $\{a',b',c'\}$, we have $\sigma* U=U'$.
 The result follow from the fact that $\sigma$ preserves distances in the tree.  
\end{proof}

\begin{lemma}
Let $a,b,c,d,d'\in\mathbb{P}_1(k)$, disposed as in Figure \ref{figmow}. If $u=min\{r,s,t,u\}\geq0$, then 
$[a,b;c,d] \equiv d'\,(\mathrm{mod}\ \pi^{l+u})$.   \label{lemrc}
\end{lemma}

\begin{figure}
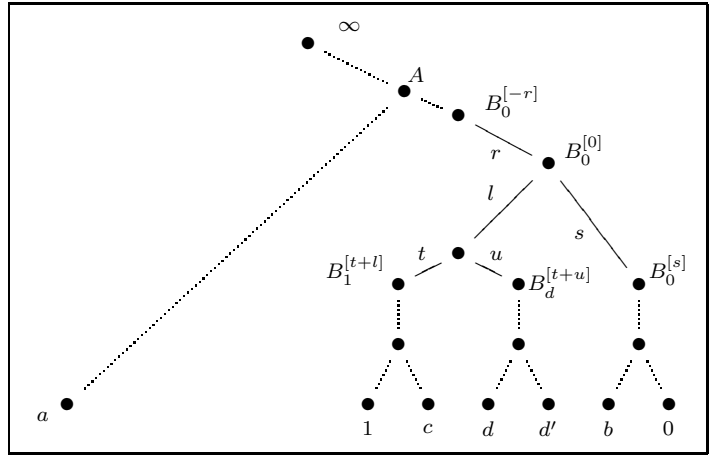

\[ \fbox{ \xygraph{
!{<0cm,0cm>;<.8cm,0cm>:<0cm,.8cm>::} 
!{(5,21) }*+{\bullet}="infty"   !{(5.7,21.2) }*+{{}^{\infty}}="inftyn"
!{(6.6,20.2) }*+{\bullet}="sep"   !{(6.8,20.4) }*+{{}^{A}}="sepn"
!{(7.5,19.8) }*+{\bullet}="A"   !{(8.4,20) }*+{{}^{B_0^{[-r]}}}="An"
!{(1,15) }*+{\bullet}="a"   !{(0.6,14.8) }*+{{}_{a}}="an"
!{(7.5,17.5) }*+{\bullet}="Z"
!{(9,19) }*+{\bullet}="U"   !{(9.6,19.2)}*+{{}^{B_0^{[0]}}}="Un"
!{(6.5,17) }*+{\bullet}="C"   !{(5.8,17.2)}*+{{}^{B_1^{[t+l]}}}="Cn"
!{(8.5,17) }*+{\bullet}="D"   !{(9.2,17)}*+{{}^{B_d^{[t+u]}}}="Dn"
!{(10.5,17) }*+{\bullet}="B"  !{(11,17.2)}*+{{}^{B_0^{[s]} }}="Bn"
!{(6.5,16) }*+{\bullet}="C2"   
!{(8.5,16) }*+{\bullet}="D2"   
!{(10.5,16) }*+{\bullet}="B2"  
!{(6,15) }*+{\bullet}="uno"   !{(6,14.6)}*+{{}_1}="unon"
!{(7,15) }*+{\bullet}="c"   !{(7,14.6)}*+{{}_c}="cn"
!{(10,15) }*+{\bullet}="b"   !{(10,14.6)}*+{{}_b}="bn"
!{(8,15) }*+{\bullet}="d"   !{(8,14.6)}*+{{}_d}="dn"
!{(9,15) }*+{\bullet}="dp"   !{(9,14.6)}*+{{}_{d'}}="dpn"
!{(11,15) }*+{\bullet}="cero"  !{(11,14.6) }*+{{}_0}="ceron"
"infty"-@{.}"sep" "sep"-@{.}"A"
 "A"-_r"U" "U"-_s"B" "U"-_l"Z" "Z"-_t"C" "D"-_u"Z" "D2"-@{.}"d" "D2"-@{.}"dp" "C2"-@{.}"uno" 
"a" -@{.} "sep" "B2"-@{.} "cero"  "C2"-@{.} "c" "B2"-@{.} "b"  
"C2"-@{.} "C" "B2"-@{.} "B" "D2"-@{.} "D"
} }
\]
\caption{Dotted lines joining vertices might have length 0, for instance, $A$ and $B_0^{[-r]}$ may coincide.
In the figure we can read, for instance, $\rho(a) \geq \rho(\pi)^{-r}$, while $d\equiv 1\ (\mathop{\mathrm{mod}} \pi^{l})$  but
$d\not\equiv 1\ (\mathop{\mathrm{mod}} \pi^{l+1})$ .} 
\label{figmow}
\end{figure}

\begin{proof}
It is apparent from Figure 5 that $\rho(c-b)=1$, whence
\begin{equation}
\frac{b-d}{b-c}=d+d\frac{1-c}{c-b} + b\frac{d-1}{c-b} \equiv d, \; (\mathrm{mod} \; \pi ^{l+u}). \label{idl1}
\end{equation}
Assume $a \neq \infty$. Since $\rho(a) \geq \rho(\pi ^{-r})$ and $c \equiv d$ (mod $\pi^l$) we obtain $a^{-1}c \equiv a^{-1}d$ (mod 
$\pi^{l+r}$) 
and 
therefore they are congruent modulo $\pi^{l+u}$. This implies that
\begin{equation}
 1 \equiv \frac{1-a^{-1}d}{1-a^{-1}c} \; (\textnormal{mod} \; \pi^{l+u}). \label{idl2}
\end{equation}
Combining (\ref{idl1}) and (\ref{idl2}) we get  
\begin{equation}
[a,b;c,d] =\frac{(a-c)}{(a-d)} \frac{(b-d)}{(b-c)} = \frac{(1-a^{-1}c)}{(1-a^{-1}d)} \frac{(b-d)}{(b-c)} \equiv d, \; (\mathrm{mod} \; \pi^{l
+u}).
\end{equation}

When $a = \infty$, $[\infty,b,c,d]=(b-d)/(b-c)$ which is congruent to $d$ modulo $\pi^{l+u}$ by (\ref{idl1}).
Since $d$ and $d'$ are congruent modulo $\pi^{l+u}$, the lemma is proved. 
\end{proof}
Next two results are consequences of the lemma. 

\begin{proposition}
Let $k$ a local field. Let $S=(A,B, C, D)$ be a quartet of balls 
whose hull is like Figure $\textnormal{\ref{figmob}}(ii)$, 
with $u=min\{r,s,t,u\}\geq0$. Let $a$, $b$, $c$, and $d$ be  ends beyond them as in
Figure $\textnormal{\ref{figmob}}(ii)$.
 Let $S'=(A',B', C', D')$ be another quartet of balls whose 
hull is isomorphic to the preceding one, while $a'$, $b'$, $c'$, and $d'$
are analogous ends. 
Then, they are conjugated if and only if 
$[a,b;c,d] \equiv [a',b';c',d'] \; (\mathrm{mod} \; \pi^{u+l})$. \label{thcr} 
\end{proposition}

\begin{proof}
Let $\tilde{\mathfrak{h}}=(\mathfrak{h},S)$ and $\tilde{\mathfrak{h}'}=(\mathfrak{h}',S)$ be 
the hulls of $S$ and $S'$, respectively.
Suppose that $S$ and $S'$ are conjugated by $\sigma \in \mathbb{M}$. Replacing 
if necessary $s=(a,b,c,d)$ by $\sigma(s)$ we can assume $\tilde{\mathfrak{h}}=\tilde{\mathfrak{h}'}$. 
Since the cross ratio is invariant under M\"obius transformations, we can assume that $a'=\infty,b'=0,c'=1$.
Therefore, the result follows from previous lemma. 

Suppose now that $[a,b;c,d] \equiv [a',b';c',d']$ (mod $\pi^{u+l}$). Let $\xi$  and $\xi'$ be the unique
M\"obius transformations satisfying $$\xi(a,b,c)=\xi'(a',b',c') =(\infty,0,1).$$
Let $t=\xi(d)$ and $t'=\xi'(d')$. We obtain $$t=[\infty,0;1,t]=[a,b;c,d]\equiv [a',b';c',d']=[\infty,0;1,t']=t',
(\mathrm{mod }\pi^{l+u}).$$ 
The congruence implies that the hulls of $\xi* S$ and $\xi'* S'$ coincide by Figure \ref{figmow}, whence 
$S$ and $S'$ are conjugated.  
\end{proof}

\begin{cor}
Let $k$ a local field. Suppose that we have a quartet $S=(A,B,C,D)$ 
whose hull is like figure $\textnormal{\ref{figmob}}(ii)$, where $u=\min\{r,s,t,u\}$. 
Suppose we have fifth ball $D'$ such that $S$ is conjugated to the quartet $S'=(A,B,C,D')$. 
Then $D=D'$. \label{corrc} 
\end{cor}

\begin{proof}
Let $a,b,c,d\in\mathbb{P}^1(k)$ be extremes beyond $A,B,C,D$ as in figure $\textnormal{\ref{figmob}}(ii)$, and let $d'$
be a extreme beyond $D'$.
Let $\xi$  be the unique M\"obius transformation satisfying $\xi(a,b,c)=(\infty,0,1)$.
The proof of Proposition \ref{thmob} impplies
$$\xi * (A,B,C,D,D')=\left( B_0^{[-r]},B_0^{[s]},B_1^{[t+l]},B_{\xi(d)}^{[u+l]},B_{\xi(d')}^{[u+l]}\right).$$
Furthermore
$$\xi(d)=[a,b;c,d]\equiv[a,b;c,d']=\xi(d'),\quad(\mathrm{mod }\pi^{l+u}),$$
so we conclude $\xi*D=\xi*D'$ and therefore $D=D'$ as claimed.   
\end{proof}

\begin{remark}
All results in this section hold if any ball $A$, $B$, $C$, or $D$, and the corresponding primed version, is
replaced by an end in $\mathbb{P}_1(k)$. For instance, if $A$ is replaced by an end we set $r=\infty$.
This observation is used in the sequel without further ado.
\end{remark}


\section{Computing embedding numbers via cross-ratio}

The results in \S5, specially Corollary 5.1, can be used to compute embedding numbers for $\Ha$ into $\Ea$,
 in the case in which
$\Ha$ and $\Ea$ are intersections
of maximal orders. This is so since a thick path is fully determined by its stem and its depth, while an infinite
leaf is fully determined by a long path.  
 Note that the group $\Gamma_1$ defined in \S1
 is precisely the group of moebius transformations fixing $S_0(\Ea)$ point-wise, while
$\Gamma_2$ is the group fixing $S_0(\Ea)$ as a set.
We use these observations in all that follows.

\begin{figure}
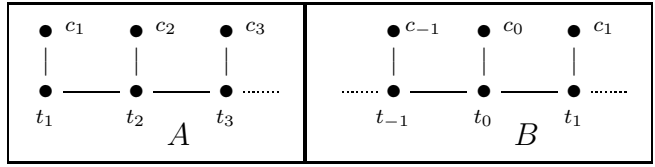

\[ \fbox{ \xygraph{
!{<0cm,0cm>;<.8cm,0cm>:<0cm,.8cm>::} 
!{(1,1) }*+{\bullet}="t1" !{(1,0.5) }*+{{}^{t_1}}="t1n"
!{(2.5,1) }*+{\bullet}="t2" !{(2.5,0.5) }*+{{}^{t_2}}="t2n"
!{(4,1) }*+{\bullet}="t3" !{(4,0.5) }*+{{}^{t_3}}="t3n"
!{(5,1) }*+{}="t6" 
!{(1,2) }*+{\bullet}="a1" !{(1.5,2) }*+{{}^{c_1}}="a1n"
!{(2.5,2) }*+{\bullet}="a2" !{(3,2) }*+{{}^{c_2}}="a2n"
!{(4,2) }*+{\bullet}="a3" !{(4.5,2) }*+{{}^{c_3}}="a3n"
!{(3.2,0.3) }*+{A}="n" 
"t1"-"t2" "t2"-"t3" "t3"-@{.}"t6"
"t1"-@{--}"a1" "t2"-@{--}"a2" "t3"-@{--}"a3"
 } }
\fbox{ \xygraph{
!{<0cm,0cm>;<.8cm,0cm>:<0cm,.8cm>::} 
!{(1,1) }*+{}="t1"
!{(2,1) }*+{\bullet}="t2" !{(2,0.5) }*+{{}^{t_{-1}}}="t2n"
!{(3.5,1) }*+{\bullet}="t3" !{(3.5,0.5) }*+{{}^{t_0}}="t3n"
!{(5,1) }*+{\bullet}="t4" !{(5,0.5) }*+{{}^{t_1}}="t4n"
!{(6,1) }*+{}="t6" 
!{(2,2) }*+{\bullet}="a2" !{(2.5,2) }*+{{}^{c_{-1}}}="a2n"
!{(3.5,2) }*+{\bullet}="a3" !{(4,2) }*+{{}^{c_0}}="a3n"
!{(5,2) }*+{\bullet}="a4" !{(5.5,2) }*+{{}^{c_1}}="a4n"
!{(4.2,0.3) }*+{B}="n" 
"t1"-@{.}"t2" "t2"-"t3" "t3"-"t4" "t4"-@{.}"t6"
 "t2"-@{--}"a2" "t3"-@{--}"a3" "t4"-@{--}"a4"
 } }
\]
\caption{The simplified diagram of $\mathfrak{S}_0(\Ha)$ in the proof of 
Theorem \ref{tb3} and Theorem \ref{tb1}.} 
\end{figure}

\paragraph{\textbf{Proof of Theorem \ref{tb1}}}
Let $L=k\otimes_{\oinki} \Ha\cong k\times k$.
 An embedding of
$\phi:L\rightarrow\matrici_2(k)$ is totally determined by an orientation of the maximal path  
$\mathfrak{S}_0\big(\phi(\oink_L)\big)$. The non-trivial automorphism of $L$ reverses this orientation.
If $n$ is the cardinality of the set $P$ of paths of length $r>0$  in $\mathfrak{S}_0(\Ha)$, without backtracking,
starting from an optimal vertex $c_0$, while $n'$ is the cardinality of the subset $P'$ of those paths in $P$ 
with two optimal extremes, we claim that $e_1=2(n-n')+n'=2n-n'$. For this, we recall that every path inside
$\mathfrak{S}_0(\Ha)$ has a unique vertex of maximal depth and the path of each side, if not trivial,
must go outwards through the leaves as in Figure 7.

 Let $P_r$ (respectively $P'_r$) denotes the set of reverses of the paths in $P$ (resp. $P'$).
The claim follows from three observations:
\begin{enumerate}
\item No two paths in $P$ can belong to the same orbit by Cor. 5.1,
\item every orbit must contain a path in $P\cup P_r$, by Prop. 5.1 applied to $c_0$ and the two ends of $\mathfrak{S}_0\big(\phi(\oink_L)\big)$,
\item The paths in $P$ that are equivalent to a path in $P_r$ are exactly those in $P'$, as follows from Prop. 5.1. 
\end{enumerate}
The paths can be counted using the simplified diagram in Figure 6B,
expanding the path joining $t_0$ and $c_0$ if needed.  
Note that no such path can go outwards through the leaves until
it reaches a returning point at distance $\left[\frac{r+1}2\right]$ if $r\leq2t$ or $r-t$ otherwise.
From the returning point on, any possible path is contained in $\mathfrak{S}_0(\Ha)$, and the ones
ending on an endpoint are exactly the ones  going outwards through the leaves precisely from the returning
point, except when $r\leq 2t$ is odd. In the latter case, there cannot be two endpoints in the path.
 It follows that $n$ and $n'$ are as in Table 1. The factor 2 in the last line is due to the fact that a path (walk)
reaching the stem can be continued along the stem on either side of the reaching point.

\begin{figure}
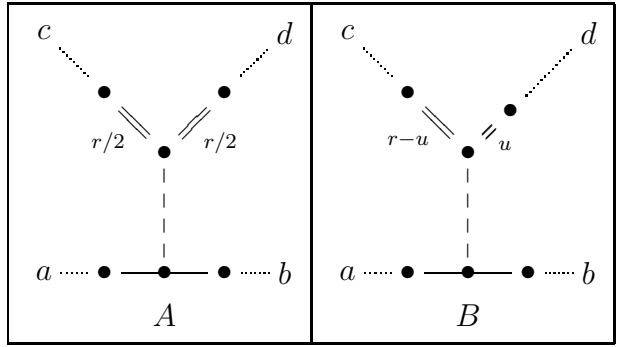

\[ \fbox{ \xygraph{
!{<0cm,0cm>;<.8cm,0cm>:<0cm,.8cm>::} 
!{(1,1) }*+{\bullet}="t1" 
!{(3,1) }*+{\bullet}="t2" !{(2,1) }*+{\bullet}="t3"
!{(0,1) }*+{a}="ti1" 
!{(4,1) }*+{b}="ti2"
!{(1,4) }*+{\bullet}="a1" !{(2,3) }*+{\bullet}="a3"
!{(3,4) }*+{\bullet}="a2" 
!{(0,5) }*+{c}="ai1" 
!{(4,5) }*+{d}="ai2"
!{(2,0.3) }*+{A}="n" 
"t1"-"t2" "t1"-@{.}"ti1" "t2"-@{.}"ti2" "t3"-@{--}"a3"
"a3"-@{=}^{r/2}"a1" "a3"-@{=}_{r/2}"a2" "a1"-@{.}"ai1" "a2"-@{.}"ai2"
 } }
\fbox{ \xygraph{
!{<0cm,0cm>;<.8cm,0cm>:<0cm,.8cm>::} 
!{(1,1) }*+{\bullet}="t1" 
!{(3,1) }*+{\bullet}="t2" !{(2,1) }*+{\bullet}="t3"
!{(0,1) }*+{a}="ti1" 
!{(4,1) }*+{b}="ti2"
!{(1,4) }*+{\bullet}="a1" !{(2,3) }*+{\bullet}="a3"
!{(2.7,3.7) }*+{\bullet}="a2" 
!{(0,5) }*+{c}="ai1" 
!{(4,5) }*+{d}="ai2"
!{(2,0.3) }*+{B}="n" 
"t1"-"t2" "t1"-@{.}"ti1" "t2"-@{.}"ti2" "t3"-@{--}"a3"
"a3"-@{=}^{r-u}"a1" "a3"-@{=}_u"a2" "a1"-@{.}"ai1" "a2"-@{.}"ai2"
 } }
\]
\caption{Location of the extremes in the computation of $e_2$ and $e_3$ in Theorem 1.4.
 Here the doble line denotes the path $\mathfrak{S}_0(\Ea)$.} 
\end{figure}

In order to compute $e_2$ we observe that any element in $\Gamma_2\backslash\Gamma_1$ interchange
the endpoints in  $\mathfrak{S}_0(\Ea)$. Every $\Gamma_2$-orbit correspond to two different 
$\Gamma_1$-orbits or to one invariant $\Gamma_1$-orbit. An orbit is invariant if the corresponding path
satisfy each of the following conditions:
\begin{itemize}
\item $r\leq2t$.
\item Both endpoints of the path are endpoints of $\mathfrak{S}_0(\Ha)$.
\item If $a$, $b$, $c$, and $d$ are extremes, located as in Figure 7A, then
$$[a,b;c,d]^{-1}=[a,b;d,c]\equiv [a,b;c,d]\ \ \  (\mathrm{mod}\ \pi^t).$$
\end{itemize}
The second condition implies that $r$ is even. Then $\chi_2$, as defined in \S1,
 is the number of invariant orbits. We obtain
$e_2=\frac12(e_1-\chi_2)+\chi_2=\frac12(e_1+\chi_2)$.
 Analogously, we prove
$e_3=\frac12(e_1+\chi_3)$, under similar conventions, since interchanging the edges of the infinite paths
$\mathfrak{S}_0\big(\phi(\oink_L)\big)$ leave invariant the paths satisfying $r\leq2t$ and 
$[b.a;c,d]\equiv [a,b;c,d]\ (\mathrm{mod}\ \pi^{t-r+2u})$, if $u$ is as in Figure 7B, 
where we no longer requires that the
path has two optimal endpoints, although this is necessarily so if $r=2t$.

To compute $e_4$ we observe that the orbits in $Y/\Gamma_1$ that remain invariant when we shift the endpoints of the path $\mathfrak{S}_0(\Ea)$ are precisely the ones corresponding to paths with two optimal endpoints, 
since the cross ratio 
has the symmetry
$[a,b;c,d]=[b,a;d,c]$. Assume $r\leq 2t$. By repeating the previous argument,
we observe that the number of orbits in $Y/\Gamma_1$  corresponding to paths with two optimal endpoints, is
$n''=\frac12(n'+\chi_2)$. We conclude that
$$e_4=\frac12(e_3+n'')=\frac14(e_1+n')+\frac14(\chi_2+\chi_3)=\frac n2+\frac14(\chi_2+\chi_3).$$
The result follows. If $r=0$ the result follows from Proposition 5.1 and the remark at the end of \S5.
\qed

\begin{figure}
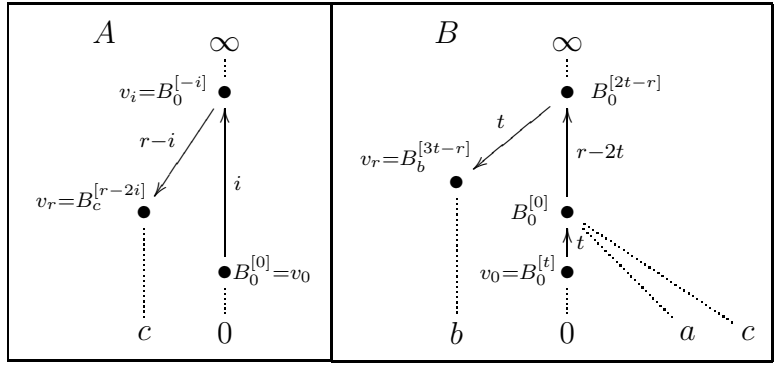

\[ 
\fbox{ \xygraph{
!{<0cm,0cm>;<.8cm,0cm>:<0cm,.8cm>::} 
!{(0.66,0) }*+{c}="e1" !{(2,0) }*+{0}="e2" !{(2,4.8) }*+{\infty}="e3"
!{(0.66,2) }*+{\bullet}="t1" !{(2,4) }*+{\bullet}="t2"  !{(2,1) }*+{\bullet}="t3"
!{(-0.2,2.2) }*+{{}^{v_r=B_c^{[r-2i]}}}="ti1" !{(1,4) }*+{{}^{v_i=B_0^{[-i]}}}="ti2" !{(2.8,1) }*+{{}^{B_0^{[0]}=v_0}}="ti3"
!{(0,5) }*+{A}="ti99"
"t1"-@{<-}^{r-i}"t2" "t2"-@{<-}^i"t3" "t1"-@{.}"e1" "t2"-@{.}"e3" "t3"-@{.}"e2"
 } }
\fbox{ \xygraph{
!{<0cm,0cm>;<.8cm,0cm>:<0cm,.8cm>::} 
!{(0.16,0) }*+{b}="e1" !{(2,0) }*+{0}="e2" !{(2,4.8) }*+{\infty}="e3"
!{(4,0) }*+{a}="e4" !{(5,0) }*+{c}="e5"
!{(0.16,2.5) }*+{\bullet}="t1" !{(2,4) }*+{\bullet}="t2"  !{(2,2) }*+{\bullet}="t3"  !{(2,1) }*+{\bullet}="t4"
!{(-0.5,2.9) }*+{{}^{v_r=B_b^{[3t-r]}}}="ti1" !{(3,4) }*+{{}^{B_0^{[2t-r]}}}="ti2" !{(1.4,2) }*+{{}^{B_0^{[0]}}}="ti3"
!{(1.2,1) }*+{{}^{v_0=B_0^{[t]}}}="ti4"
!{(0,5) }*+{B}="ti99"
"t1"-@{<-}^t"t2" "t2"-@{<-}^{r-2t}"t3" "t3"-@{<-}^t"t4" "t1"-@{.}"e1" "t2"-@{.}"e3" "t4"-@{.}"e2"
"t3"-@{.}"e4" "t3"-@{.}"e5"
 } }
\]
\caption{Configuration of paths in the proofs of Theorem 1.2 (A) and Theorem 1.3 (B). The branch of
$\Ea$ is denoted with arrows. Note that in B, the integer
$a$ or $c$ might be congruent to $0$ modulo $\pi$, buy $a$ and $c$ are not congruent to each other.} 
\end{figure}

\paragraph{\textbf{Proof of Theorem \ref{tb2}}}
If $\Ha$ is generated by a nilpotent element $\nu$, then $\mathfrak{S}=\mathfrak{S}_0(\Ha)$ is an infinite
 leaf (see \S1).
If $\Ha$ is optimal in $\Ea$, then one of the endpoints $v_0$ of the path $\mathfrak{S}_0(\Ea)$ is an endpoint of 
$\mathfrak{S}$. Let $v_0,v_1,\dots,v_r$ be the vertices of $\mathfrak{S}_0(\Ea)$ in order.
There is a unique vertex $v_i$ whose relative depth \cite[\S2]{eichler4}
$p(v_i,\mathfrak{S})=p_0$ is maximal, and the path from each side of $v_0$ goess outwards through the leaves.
It follows that $p_0=i$ and $0\leq r-i\leq [r/2]$. To compute $e_4$,
it suffices to prove that $r-i$ completely determines the $\Gamma_2$-orbit of an isomorphic copy of $\Ha$,
or equivalently that $r-i$ completely determines the $\mathbb{M}$-orbit of the pair $(\Ha,\Ea)$,
applying Prop. \ref{triplets} to the triplet formed by the endpoints of  $\mathfrak{S}_0(\Ea)$
and the common end of all long paths in the infinite leaf. To compute $e_3$ we observe that a
$\Gamma_2$-orbit correspond to exactly one  $\Gamma_1$-orbit precisely when both
extremes of $\mathfrak{S}_0(\Ea)$ are endpoints of $\mathfrak{S}_0(\Ha)$, reasoning 
as in the preceding proof.

To compute $e_1$ and $e_2$, fix an optimal embedding $\phi:\Ha\rightarrow\Ea$, so that 
$v_0$ is an extreme of $\mathfrak{S}_0\Big(\phi(\Ha)\Big)$, and assume $v_i$  is the deepest vertex as before.
Using the $\mathbb{M}$-action, we can assume that $\mathfrak{S}_0\Big(\phi(\Ha)\Big)$ is the graph
whose vertices correspond to all balls of radius 1 or larger as in Figure 8A.
$v_0$ is the vertex corresponding to $B_0^{[0]}$, $v_i$ is the vertex corresponding to $B_0^{[-i]}$, 
and $v_r$ is the vertex corresponding to $B_c^{[-2i+r]}$, with $\rho(c)=\rho(\pi)^{-i}$. 
Furthermore, we have $\phi(\nu)=\lbmatrix 0u00$ for some $u\in\oinki^*$.
In particular, conjugation by $\phi(1+cu^{-1}\nu)$ sends the end $0$ to the end $c$, whence, if $i$ is fixed, the class
of $cu^{-1}$ modulo $\pi^{r-2i}$, or equivalently the class 
$\overline{cu^{-1}\pi^{i}}\in(\oinki/\pi^{r-i}\oinki)^*$, is a complete invariant of the conjugacy class 
of the trio $(v_0,v_r,\phi)$. We conclude that the total number of possible invariants is
$$n=1+(p-1)\sum_{r-i=1}^{[r/2]}p^{r-i-1}=p^{[r/2]},$$
while the number of invariants corresponding to paths $\mathfrak{S}_0(\Ea)$ with two endpoints is
 $n'=(p-1)p^{r/2-1}$ if $r$ is even, and $n'=0$ when $r$ is odd.
Then $e_1=2n-n'$ as for Theorem \ref{tb1}. 
When $i=r/2$, the Moebius transformation $z\mapsto z-c$ leaves $\phi$ invariant while it
sends a path whose invariant is $\overline{cu^{-1}\pi^i}$ to a path whose invariant is  $-\overline{cu^{-1}\pi^i}$.  
Since $\overline{cu^{-1}\pi^i}$ is a unit, they are equal exactly when $2\equiv 0\ (\mathrm{mod}\ \pi^{r/2})$.
If the latter condition holds, every path in $P'$ is equivalent to a path in $P'_r$, if $P'$ and $P'_r$ are defined as in the preceding proof. We conclude that $e_2=(n-n')+n'=n$. If the condition fails to hold, or if $r$ is odd,
we have $e_2=e_1/2$. 
\qed

\paragraph{\textbf{Proof of Theorem \ref{tb3}}}
By Proposition \ref{thmob} and Lemma \ref{newlemma}, we can always assume that the embedding $\phi:L\rightarrow
\alge$ and the Eichler order $\Ea$ are choosen in a way that the stem $T$ of the thick ray $R=\mathfrak{S}_0\big(\phi(\Ha)\big)$ and the path $C=\mathfrak{S}_0(\Ea)$ are like in figure 8B, where $T$ is the ray joining $B_0^{[0]}$ and $\infty$.
 Since $(S,C\cup T)$ is the hull of $S=\left(\infty,B_0^{[t]},B_b^{[3t-r]}\right)$, all statements about $e_3$ and $e_4$ follows
as before. Note that the endpoints of $C$ can be shifted by an element stabilizing $\Ha$ if and only if $r=2t$.

Recall that an embedding $\phi:L\rightarrow \matrici_2(k)$ is completely determined by the images of the matrices
$\mu=\lbmatrix 1000$ and $\mu'=\lbmatrix 1100$. They correspond to two maximal path 
with a common end, but not the other. The intersection of these two paths is the ray $S_0\Big(\phi(\Ha_0)\Big)$,
where $\Ha_0=\lbmatrix {\oinki}{\oinki}0{\oinki}$. Assume now that this is the ray with end $\infty$ and stem border
$B_0^{[0]}$. Then $\phi$ is completely determined by a pair of ends $(a,c)\in \mathbb{P}^1(k)$, which are integers with
different images in the residue field, because of the condition that the intersection of the two maximal paths is the ray. 
Reasoning as in the preceding proof and setting $d=b\pi^{2t-r}\in\oink_k^*$,  we conclude that the pair
$$\left(\overline{ad^{-1}},\overline{cd^{-1}}\right)\in\left\{(x,y)\in(\oinki/\pi^t\oinki)^2 \Big|x\not\equiv y
\ \ (\mathop{\mathrm{mod}} \pi)\right\}$$
 is a complete invariant, 
and the result follows as before (see Figure 8B). The relation $e_4=\frac12e_3$ is trivial if $r\neq 2t$.
If $r=2t$, we observe that the Moebius
 transformation $z\mapsto b-z$ interchanging the endpoints of
$C$ replaces the invariant $(x,y)$ by $(1-x,1-y)$, which cannot be equal to $(x,y)$ as 
the equation $x=1-x$ has one solution in $\oinki/\pi^t\oinki$ when $2$ is invertible, and none otherwise,
but never $2$ distinct solutions.
\qed

\section{Ackonwledgements}

The first author was supported by Fondecyt, grant No 1120844.
The second author was supported by Fondecyt, grant No 1140533, while the third
author was partly supported by Fondecyt, grant No 1120565.

\end{document}